\documentclass[11pt]{amsart}
\usepackage{amsmath,amssymb,amsthm,mathrsfs}
\usepackage[colorlinks=true,citecolor=blue,linkcolor=blue,urlcolor=blue]{hyperref}
\usepackage[left=0.82in,right=0.82in,top=0.55in,bottom=0.62in]{geometry}

\newtheorem{theorem}{Theorem}[section]
\newtheorem{proposition}[theorem]{Proposition}
\newtheorem{corollary}[theorem]{Corollary}

\theoremstyle{remark}

\newcommand{\Z}{\mathbb Z}
\newcommand{\Q}{\mathbb Q}
\newcommand{\Pone}{\mathbb P^1}
\newcommand{\Mod}{\operatorname{Mod}}

\title[Cartwright--Steger fibration]{A monodromy relation for the Cartwright--Steger fibration}
\author{Anar Akhmedov}
\address{School of Mathematics, University of Minnesota, Minneapolis, MN 55455, USA}
\author{Sai-Kee Yeung}
\address{Department of Mathematics, Purdue University, West Lafayette, IN 47907, USA}
\date{September 2026}

\begin{document}
\begin{abstract}
We study the genus-$19$ Albanese fibration of the Cartwright--Steger surface and its order-three symmetry. Passing to the orbifold quotient of the elliptic base gives a relation among the two handle monodromies and the three Dehn twists about the vanishing cycles. The two real vanishing paths give cycles whose real fixed points lie on different ovals of the invariant fiber. For a cyclic triple cover, we express the intersection of a curve with its image under the deck transformation as a signed crossing count on the quotient. We determine the two local branch values of a degree-$72$ bicanonical pencil near a node and study a symmetric bicanonical pencil. A general pencil also gives a finite degree-$72$ map from a blowup of the surface to the product of the elliptic curve and a projective line. The monodromy translates of the nodal vanishing classes span the first homology of a smooth fiber.
\end{abstract}
\maketitle

\section{Introduction}

Let $X$ be the Cartwright--Steger surface and $\alpha:X\to E$ its Albanese map. The smooth fibers have genus $19$. The surface has
\[
e(X)=3,\qquad \operatorname{sign}(X)=1,\qquad K_X^2=9=3c_2(X).
\]
Koziarz and Yeung proved that the Albanese fibration is stable \cite{KY}. Independently, Rito showed that the three fibers fixed by the order-three automorphism $\sigma$ are smooth, which also implies stability in view of \cite{CKY,Rito}. All singular points therefore occur in orbits of three. For a nodal fibration over a torus, $e(X)$ counts the nodes, so $e(X)=3$ gives one such orbit: three nodes in three distinct fibers. In this three-node case, the singular fibers are irreducible by \cite[long version, Proposition~5(iii)]{CKY}.
Cartwright, Koziarz, and Yeung described the Albanese fibration \cite{CKY}; Borisov and Yeung gave equations for $X$ \cite{BY}.

Monodromy describes how a smooth fiber changes as it is transported around a loop in the base. Together with the vanishing cycles at the singular fibers, it records how the surface is assembled from its fibers. For a Lefschetz fibration over a disk, each vanishing cycle specifies a $2$-handle attached to the product of a smooth fiber with a disk, with framing one less than the fiber framing. For a base of positive genus, the handle monodromies also describe the gluing around the generators of the base. Thus a marked monodromy description gives concrete information for a handlebody description of the total space. It can be used to compute the fundamental group and study embedded surfaces needed for symplectic sums and surgery constructions. Such constructions have produced exotic smooth structures, that is, smooth four-manifolds which are homeomorphic but not diffeomorphic to the standard examples, with small Euler characteristic and second Betti number. The first author's earlier work gives examples of this kind \cite{AkhSmall,APSmall,APOdd}. Akhmedov, Sakall\i{}, and Yeung subsequently used the Cartwright--Steger surface and other complex ball quotients to construct symplectic four-manifolds with nonnegative signature \cite{ASY}.

The small characteristic numbers of $X$ make its Albanese fibration particularly interesting in this context. Our recent work on the real Albanese map \cite{AYreal} and on positive spheres and circle surgery in the real quotient \cite{AYquotient} raises related questions about embedded surfaces and smooth structures. A marked monodromy description could provide further information about the surfaces available for these constructions and the relations imposed on their fundamental groups. The monodromy description developed here provides a tool for constructing exotic four-manifolds with small characteristic numbers. We will pursue these applications in forthcoming work.

Our aim is to make the monodromy problem for this Albanese fibration more explicit. A marked description requires both the vanishing cycles on a reference fiber and the monodromies around the handle generators of the punctured elliptic base. We study these data using the real invariant fiber $F$, its genus-six cyclic quotient $G$, and the order-three quotient of the base. Theorem~\ref{thm:crossing} gives a signed-crossing formula for the intersection of a cycle with its deck translate. Proposition~\ref{prop:schreier} and Theorem~\ref{thm:marked} give the monodromy relation and the identities arising from the order-three symmetry; Proposition~\ref{prop:local-branch-pair} identifies the local branch pair at the real node. Proposition~\ref{prop:symmetric-pencil} and Theorem~\ref{thm:product-cover} give a symmetric pencil and a finite product cover that may be used for global continuation. Theorem~\ref{thm:vanishing-span} shows that all monodromy translates of the vanishing classes span the homology of a smooth fiber. Proposition~\ref{prop:albanese-kernel-normal-generation} gives normal generators for the kernel of the Albanese homomorphism. These results supply a starting point for determining the marked monodromy. They do not yet determine the isotopy class of the transported vanishing cycle or the handle monodromies. We are not aware of a comparably explicit marked monodromy description for the Albanese fibration of a complex ball quotient surface.

The two-oval calculation for the real invariant fiber is in \cite{AYreal}. The unique real critical point and its saddle structure are described in \cite{SS}. Szab\'o computed the Albanese map and its critical points explicitly \cite{Szabo}; we use the critical value to specify a real path and the local Hessian at the node.

\section{The invariant fiber and its quotient}

Identify $E$ with $\mathbb C/(\mathbb Z+\omega\mathbb Z)$, where $\omega=e^{2\pi i/3}$, as in \cite{CKY,AYreal}. Let $O=0$. We choose the generator $\sigma$ so that its action on $E$ is
\[
\sigma_E(z)=\omega^2z,\qquad \alpha\sigma=\sigma_E\alpha.
\]
The real structure fixes $O$, and $O$ is the only point of $E$ fixed by both the real involution and $\sigma_E$ \cite[Lemma~7]{AYreal}. Put $F=\alpha^{-1}(O)$ and $S=\sigma|_F$. The fiber $F$ is smooth of genus $19$ \cite{KY,Rito}. In the Borisov--Yeung bicanonical coordinates $[U_0:\cdots:U_9]$, the three fixed points of $S$ are the real coordinate points
\[
 p_7=[U_7],\qquad p_8=[U_8],\qquad p_9=[U_9].
\]
The quotient map $\pi:F\to G=F/\langle S\rangle$ has degree three and is totally ramified at these three points. Riemann--Hurwitz gives
\[
2\cdot19-2=3(2g(G)-2)+3(3-1),\qquad g(G)=6.
\]
The real structure $C$ on $F$ satisfies $CSC=S^{-1}$. The restriction $\pi:F(\mathbb R)\to G(\mathbb R)$ is a homeomorphism, and $F(\mathbb R)$ has two ovals \cite{AYreal}.

The same work establishes that
\[
L_0=U_7+U_8+U_9,\qquad L_1=U_7+2U_8+3U_9
\]
have no common zero on $F$. Their ratio defines a degree-$72$ map $\ell:F\to\Pone$. Cartwright, Koziarz, and Yeung first computed $K_X\cdot F=36$ and obtained $g(F)=19$ by adjunction \cite[long version, Theorem~3]{CKY}. Thus, for the bicanonical class $H=2K_X$, we have $H\cdot F=72$. The base-point-freeness of this particular pair on $F$ and the resulting degree computation are proved in \cite[Section~5, Proposition~10]{AYreal}, using the bicanonical model of \cite{BY}. The map is $S$-invariant and descends to a degree-$24$ map $\bar\ell:G\to\Pone$ \cite[Proposition~10]{AYreal}. Its ramification divisor has degree $58$.

\begin{proposition}\label{prop:three-labels}
In the affine coordinate $L_1/L_0$ on the target of $\bar\ell$, the three branch points of $\pi$ map to $1$, $2$, and $3$, respectively.
\end{proposition}
\begin{proof}
At $p_7$, $p_8$ and $p_9$, respectively, the pairs $(L_0,L_1)$ are $(1,1)$, $(1,2)$ and $(1,3)$ in the corresponding homogeneous coordinate normalizations. Since $\ell=\bar\ell\pi$, the assertion follows.
\end{proof}

Both sections transform by the same character of $\sigma$, so the rational function $\ell=L_1/L_0$ satisfies $\ell\sigma=\ell$ on $X$. The automorphism permutes general fibers of $\alpha$; its restriction $S$ acts on the invariant fiber $F$.

\begin{proposition}\label{prop:global-pencil}
The base locus of the pencil spanned by $L_0,L_1$ is a zero-dimensional scheme of length $36$. On every smooth fiber disjoint from the base locus, $\ell=L_1/L_0$ defines a morphism of degree $72$.
\end{proposition}
\begin{proof}
Let $B=\{L_0=L_1=0\}\cap X$. Proposition~\ref{prop:three-labels} and the base-point-freeness quoted above show that $B$ is disjoint from $F$. If $B$ contained a curve $D$, then $D$ would be disjoint from $F$, hence $D\cdot F=0$ and $D$ would be vertical for $\alpha$. All smooth fibers are irreducible, and the three singular fibers are irreducible and reduced \cite[long version, Proposition~5(iii)]{CKY}; hence the support of $D$ would be a full fiber $F_t$. A hyperplane divisor containing $F_t$ would have the form $H=F_t+R$ with $R$ effective. Since $H$ is ample, this would give
\[
36=H^2=H\cdot F_t+H\cdot R\geq 72,
\]
a contradiction. Thus the two hyperplane divisors meet properly. Their intersection scheme has length $H^2=36$. On a fiber avoiding $B$, their restrictions have no common zero and define a map of degree $H\cdot F_t=72$.
\end{proof}

\section{A real vanishing path}

Let $b\in E(\mathbb R)$ be the critical value of the real Albanese map. There is exactly one such value, and the corresponding critical point is a real saddle \cite{SS}. Choose one of the two real arcs from $O$ to $b$ whose interior avoids $b$, and denote it by $\gamma$. Since the other two critical values are the $\sigma_E$-translates of $b$, the paths $\gamma$, $\sigma_E\gamma$, and $\sigma_E^2\gamma$ connect the three nodes to the common smooth reference fiber $F$. Let $c\subset F$ be the vanishing cycle of $\gamma$. With equivariant transport, the other two cycles are $S(c)$ and $S^2(c)$.

The arc may be specified in the Weierstrass coordinate of \cite{Szabo}, where $E$ is $y^2=x^3-3888$ and $O$ is the point at infinity. Put
\[
 y_* = \frac{487135434066}{1459^3}.
\]
For the remainder of the paper we take $\gamma$ to be the positive-$x$ real arc
\[
 \gamma(y)=\bigl((y^2+3888)^{1/3},y\bigr),\qquad y:+\infty\longrightarrow y_*.
\]
It begins at $O$ and ends at the real critical value computed in \cite{Szabo}; its interior contains no critical value. This choice fixes which of the two real paths is used for $c$ and for the connecting path of the orbifold meridian $d$.

The path $\gamma$ is preserved by the real involution. Real holomorphic Morse coordinates at the node give a transverse real crossing, locally of the form $xy=\delta$ in real coordinates, and a small real smoothing has a conjugation-invariant vanishing cycle with two real fixed points. Transport along $\gamma$ may be chosen to commute with conjugation. Thus $C(c)=c$ as an unoriented curve, while $C$ reverses its orientation. In particular,
\[
C_*[c]=-[c]\in H_1(F;\mathbb Z).
\]
This local description does not determine $[c]$ in a basis of $H_1(F)$.

There are two choices of real arc from $O$ to $b$. Their real vanishing cycles have different positions relative to the two ovals of $F$.

\begin{proposition}\label{prop:two-real-arcs}
Let $c_+$ and $c_-$ be transported along the two real arcs from $O$ to $b$, using conjugation-equivariant transport. The two real fixed points of each $c_\pm$ lie on one oval of $F(\mathbb R)$, and the ovals obtained for $c_+$ and $c_-$ are different.
\end{proposition}
\begin{proof}
Cut the real base circle at $O$. The connectedness of $X(\mathbb R)$ and the decomposition of the cut real fibration into a saddle component and a product annulus follow from \cite[Theorem~1 and Corollary~2]{AYreal}; the unique real saddle is also described in \cite{SS}. Each regular real fiber has two circles. The saddle component is diffeomorphic to $\mathbb{RP}^2\setminus(D^2\sqcup D^2)$. The saddle component has one boundary circle on either side. In the real Morse model, the two fixed points of the local vanishing circle lie on the regular level circle of this saddle component; the two local arcs join outside the Morse chart to form that circle. Equivariant transport along either arc carries the fixed points to the corresponding boundary circle of the cobordism. The boundary-gluing permutation is the transposition for two components by \cite[Theorem~1]{AYreal}: otherwise the reglued real surface would be disconnected. Hence the boundary circles of the saddle component on the two sides are identified at $O$ with different ovals of $F(\mathbb R)$.
\end{proof}

The proposition specifies an oval only after a path has been chosen. It does not determine the isotopy class of either $c_\pm$ on $F$ or the parity of its intersection with $S(c_\pm)$.

\section{Intersections on a cyclic triple cover}

We first work with a cyclic triple cover of closed oriented surfaces. For Riemann surfaces, intersections are taken with the complex orientation.

\begin{theorem}\label{thm:crossing}
Let $\pi:F\to G$ be a cyclic degree-three branched cover with deck generator $S$ and branch set $D\subset G$. Let $c\subset F$ be an oriented embedded circle disjoint from $\pi^{-1}(D)$. After a small perturbation, assume $\beta=\pi|_c$ is an immersion with transverse double points and no triple points. At a double point $\beta(t)=\beta(u)$, $t<u$, let
\[
c(u)=S^{j(t,u)}c(t),\qquad j(t,u)\in\{1,2\},
\]
and let $\varepsilon(t,u)$ be the local intersection sign of the ordered tangent vectors $(\dot\beta(t),\dot\beta(u))$ on $G$. Then
\begin{equation}\label{eq:crossing}
[c]\cdot S_*[c]
=\sum_{\substack{t<u\\\beta(t)=\beta(u)}}
\begin{cases}
-\varepsilon(t,u),&j(t,u)=1,\\
\phantom{-}\varepsilon(t,u),&j(t,u)=2.
\end{cases}
\end{equation}
Here the loop $\beta|_{[t,u]}$ is based at the double point. Our monodromy convention lifts it from $c(t)$ to $S^{j(t,u)}c(t)$.
\end{theorem}
\begin{proof}
At a double point the two lifts are distinct because $c$ is embedded; hence $j$ is $1$ or $2$. If $j=1$, $c$ meets $S(c)$ at $c(u)=S(c(t))$. The ordered tangents of $(c,S(c))$ project to $(\dot\beta(u),\dot\beta(t))$, so the intersection sign is $-\varepsilon$. If $j=2$, then $c(t)=S(c(u))$ and the ordered projected tangents are $(\dot\beta(t),\dot\beta(u))$, giving $+\varepsilon$. The two occupied sheets form a pair in a three-element cyclic set. This pair and its translate have one element in common, so the double point accounts for exactly one intersection of $c$ with $S(c)$. Conversely, every intersection of $c$ and $S(c)$ projects to a double point of $\beta$. Summing the local signs proves \eqref{eq:crossing}.

Interchanging $t$ and $u$ reverses the local crossing sign and replaces $j$ by $-j$ modulo three. The summand is unchanged, as required.
\end{proof}

\begin{corollary}\label{cor:real-parity}
Suppose $F$ and $G$ have compatible real involutions $C$ and $J$, with $CSC=S^{-1}$, and $c$ is $C$-invariant as an unoriented circle. Choose a compatible generic perturbation in Theorem~\ref{thm:crossing}. Then
\[
[c]\cdot S_*[c]\equiv
\#\{\text{$J$-fixed double points of $\beta$}\}\pmod2.
\]
\end{corollary}
\begin{proof}
The involution $J$ preserves the set of double points of $\beta$. Nonfixed double points occur in pairs, each of whose two summands is odd modulo two. A fixed double point contributes one modulo two. Theorem~\ref{thm:crossing} gives the claim. At an unbranched real point of $G$, the involution on the three-sheet fiber fixes exactly one sheet, since it conjugates $S$ to $S^{-1}$. The two lifts on the $C$-invariant curve $c$ form a $C$-stable set of two distinct sheets. They cannot both be fixed, so $C$ exchanges them.
\end{proof}

Put $R_j=\operatorname{Fix}(S^jC)$ for $j=0,1,2$. The parity can also be computed on $F$.

\begin{corollary}\label{cor:real-translate}
Under the hypotheses of Corollary~\ref{cor:real-parity}, suppose also that $c$ meets $R_1$ transversely and avoids the fixed points of $S$. Then
\[
[c]\cdot S_*[c]\equiv [c]\cdot[R_1]\pmod2.
\]
\end{corollary}
\begin{proof}
The real involution $C$ exchanges $R_1$ and $R_2$. If $z\in c\cap R_1$, then $S C z=z$, so $Cz=S^2z$ is a distinct point of $c\cap R_2$. The pair $z,Cz$ projects to one double point fixed by the involution on $G$. Conversely, at a fixed projected double point the two branches are exchanged by $C$; of its two lifts on $c$, exactly one lies in $R_1$. The result follows from Corollary~\ref{cor:real-parity} and the mod-two intersection count.
\end{proof}

For the Cartwright--Steger fiber, isotope the transported vanishing cycle $c$ off the three fixed points. For the real parity formula this can be done $C$-equivariantly: in a small real coordinate disk at a fixed point, move any transverse intersection of $c$ with the real axis slightly along that axis, away from the fixed point, and apply the reflected local perturbation. The three branch points of $\pi$ have $\bar\ell$-values $1,2,3$ by Proposition~\ref{prop:three-labels}. Write
\[
m=[c]\cdot S_*[c].
\]
If $\beta$ is embedded, then $m=0$. The two real ovals alone give no count of the real double points of $\beta$.

The fixed loci of the three antiholomorphic involutions $S^jC$ have an intersection property independent of $c$.

\begin{proposition}\label{prop:reflection-loci}
For distinct $i,j\in\{0,1,2\}$, the curves $R_i=\operatorname{Fix}(S^iC)$ and $R_j=\operatorname{Fix}(S^jC)$ meet precisely at $p_7,p_8,p_9$, and they meet transversely there. In particular,
\[
[R_i]\cdot[R_j]=1\pmod 2.
\]
\end{proposition}
\begin{proof}
A point fixed by both $S^iC$ and $S^jC$ is fixed by their product $S^{i-j}$, hence by $S$. Conversely, the three fixed points of $S$ are real and are fixed by every $S^iC$. In a local holomorphic coordinate at any such point, $S$ acts by a nontrivial cube root of unity and $C$ is a reflection. The fixed real lines of $S^iC$ are distinct for distinct $i$, so the intersections are transverse. There are three of them, giving the stated mod-two intersection.
\end{proof}

This intersection does not give the parity of $m$: Corollary~\ref{cor:real-translate} requires the intersection of the particular vanishing cycle with $R_1$.

\section{The monodromy relation}

Write $v_i=S_*^i[c]$. Since $\sigma_E(z)=\omega^2z$, the positive cyclic order of the three paths near $O$ is $\gamma,\sigma_E^2\gamma,\sigma_E\gamma$. We use the algebraic intersection convention for which a positive twist acts by $T_v(x)=x+(v\cdot x)v$. By $S$-invariance of intersection,
\[
v_0\cdot v_1=v_1\cdot v_2=v_2\cdot v_0=m.
\]

\begin{proposition}\label{prop:trace}
For the product $M=T_{v_0}T_{v_2}T_{v_1}$ in this cyclic order, acting on $H_1(F;\Q)$,
\[
\operatorname{tr}(M)=38-3m^2-m^3,
\qquad
\det(zI-M)=(z-1)^{36}
\bigl[z^2-(2-3m^2-m^3)z+1\bigr].
\]
For the reverse product $T_{v_0}T_{v_1}T_{v_2}$, replace $m$ by $-m$ in these formulas.
\end{proposition}
\begin{proof}
Set $N_i(x)=(v_i\cdot x)v_i$. Then $T_{v_i}=I+N_i$, $\operatorname{tr}N_i=0$, and for $i\ne j$, $\operatorname{tr}(N_iN_j)=-(v_i\cdot v_j)^2=-m^2$. The trace of $N_0N_2N_1$ is $-m^3$. Expanding the product proves the trace formula. Its difference from the identity has rank at most three, so the eigenvalue $1$ has algebraic multiplicity at least $35$. For a symplectic matrix of dimension $38$ this multiplicity is even, hence at least $36$. The other two eigenvalues have sum $2-3m^2-m^3$ and product one, which gives the displayed polynomial.
\end{proof}

The nodal homology action gives a direct formula for $m$.

\begin{proposition}\label{prop:single-trace}
Let $T=T_{[c]}$ act on $H_1(F;\Q)$ by $T(x)=x+([c]\cdot x)[c]$. Then
\[
m=1+\operatorname{tr}(S_*T).
\]
Moreover,
\[
m^2=38-\operatorname{tr}(T S_*T S_*^{-1}).
\]
\end{proposition}
\begin{proof}
The three fixed points of $S$ give a Lefschetz number of $3$, hence $\operatorname{tr}(S_*|H_1(F;\Q))=-1$. The operator $S_*T-S_*$ has rank one: it sends $x$ to $([c]\cdot x)S_*[c]$, so its trace is $[c]\cdot S_*[c]=m$. For the second identity, expand the two transvections $T_{[c]}T_{S_*[c]}$. Each rank-one term has trace zero and their product has trace $-m^2$.
\end{proof}

For a Lefschetz fibration over a torus, the commutator of the handle monodromies is a product of positive Dehn twists about the vanishing cycles \cite{SY}. The order-three symmetry constrains these twists. The quotient of the elliptic base by $\sigma_E$ is a sphere with three orbifold points of order three. The three nodal values have one image in this quotient. Choose a positively oriented loop $p$ about the orbifold point corresponding to $O$, and take $d$ around the image of the nodal values along the projection of $\gamma$. Choose loops $q,r$ about the other two orbifold points, with connecting paths and order so that $pqrd=1$. The orbifold fundamental group of the quotient base and the character defining the degree-three cover are
\[
\Delta=\langle p,q,r,d\mid p^3=q^3=r^3=1,\ pqrd=1\rangle,
\qquad
\epsilon(p)=\epsilon(q)=\epsilon(r)=1,\quad\epsilon(d)=0\in\Z/3.
\]
The kernel $K=\ker\epsilon$ is the fundamental group of the torus with its three critical values deleted. We compose loops and mapping classes so that monodromy is a homomorphism. The peripheral loops are counterclockwise meridians. With these choices, the commutator of the two handle generators equals the product of the three transported meridians.

Over the sphere with its four distinguished values removed, the quotient of the smooth Albanese fibration by $\langle\sigma\rangle$ is a smooth $F$-bundle: the action on the punctured base is free. Near any of the three omitted invariant smooth fibers, equivariant local trivialization identifies the monodromy about the corresponding orbifold point with a periodic mapping class of order three. After choosing a reference fiber and connecting paths, we obtain an orbifold monodromy representation $\rho:\Delta\to\Mod(F)$. Its restriction to $K$ is the monodromy of the original fibration over the punctured elliptic base. The quotient total space can have cyclic quotient singularities over the orbifold points; the orbifold formulation accounts for them.

\begin{proposition}\label{prop:schreier}
Let
\[
A=qp^{-1},\qquad B=pqp^{-2},\qquad D_i=p^idp^{-i}\quad(i=0,1,2),
\]
and use $[x,y]=xyx^{-1}y^{-1}$. Then $K$ has the presentation
\begin{equation}\label{eq:schreier}
[A,B]=D_0(BD_2B^{-1})(BA D_1A^{-1}B^{-1}).
\end{equation}
If $\rho:\Delta\to\Mod(F)$ is the orbifold monodromy representation with $\rho(d)=T_c$, the three peripheral factors in \eqref{eq:schreier} are positive Dehn twists, with the indicated transport conjugations.
\end{proposition}
\begin{proof}
Use $1,p,p^2$ as coset representatives for $K$ in $\Delta$. Reidemeister--Schreier rewriting gives the stated generators. To check the relation explicitly, put $L=pq$. Since $q=Ap$ and $B=pAp^{-1}$, the relation $q^3=1$ implies
\[
L=Bp^2,\qquad L^2=BAp,\qquad L^3=[B,A].
\]
The relation $pqrd=1$ gives $r=L^{-1}d^{-1}$, so $r^3=1$ is equivalent to $(dL)^3=1$. Expanding this last word yields
\[
(dL)^3=D_0(BD_2B^{-1})(BA D_1A^{-1}B^{-1})[B,A].
\]
Since $[B,A]=[A,B]^{-1}$, this gives \eqref{eq:schreier}. The resulting group is a genus-one surface group with three peripheral generators, as its covering interpretation also shows. Applying $\rho$ sends each conjugate of $d$ to a conjugate of the positive twist $T_c$.
\end{proof}

The monodromy about the orbifold point corresponding to $O$ can be identified with $S$ after fixing a local marking.

\begin{theorem}\label{thm:marked}
Identify a nearby reference fiber of the quotient fibration with $F$ by equivariant radial transport toward $O$. Let $P=\rho(p)$, $Q=\rho(q)$, $R=\rho(r)$, and $T=\rho(d)=T_c$, with the above choice of the positively oriented loop $p$. Then
\[
P=S,\qquad B=SAS^{-1},\qquad Q=AS,\qquad
R=(T_cSAS)^{-1},
\]
and, in particular,
\begin{equation}\label{eq:marked-constraints}
(AS)^3=1,\qquad (SAS T_c)^3=1.
\end{equation}
\end{theorem}
\begin{proof}
Choose a complex coordinate $z$ centered at $O$ in which $\sigma_E(z)=\omega^2z$. A local coordinate on $E/\langle\sigma_E\rangle$ is $u=z^3$. A positive $u$-meridian lifts from $z$ to $\omega z=\sigma_E^2(z)$. In the quotient bundle, identification of the endpoint fiber with the initial one uses $\sigma^{-2}=\sigma$. Radial transport to $F$ therefore identifies its return mapping class with $S$. The definitions $A=QP^{-1}$ and $B=PQP^{-2}$ give $Q=AS$ and $B=SAS^{-1}$. Finally $PQRT=1$ gives $R=(SAS)^{-1}T_c^{-1}=(T_cSAS)^{-1}$. The cubic constraint in \eqref{eq:marked-constraints} follows because $T_cSAS$ is conjugate to $SAST_c$.
\end{proof}

The preceding relations also involve monodromy away from the critical values. The handle generators $A,B$ in Proposition~\ref{prop:schreier} represent the two independent cycles of the elliptic base. The generator of $\pi_1(E(\mathbb R))\cong\Z$ is represented by the real circle, which passes through the critical value $b$. Thus the real circle itself has no smooth-fiber return map. Choose a small disk about $b$ and replace the segment through $b$ by a detour in that disk, keeping the rest of the circle fixed and taking $O$ as base point. The resulting loop in $E\setminus\{b,\sigma_Eb,\sigma_E^2b\}$ projects to a primitive class in $H_1(E;\Z)$; its based class is a word in $A,B,D_i$. The two local detours differ by insertion of a based conjugate of the meridian about $b$, so their return maps differ by the corresponding transported twist. Neither the local node calculation nor the order-three relations determine either return map. Computing one requires global parallel transport along the chosen detoured loop.

The fixed-point calculation in \cite[long version, Main Theorem~(e)]{CKY} gives three fixed points on each invariant fiber. It also gives type $\frac13(1,1)$ at the three points of $F$ and type $\frac13(1,2)$ at the six other fixed points. With our choice $\sigma_E(z)=\omega^2z$, the tangent multiplier on $F$ at its three fixed points is $\omega^2$, whereas the fiber tangent multipliers at each of the other two invariant fibers are $\omega$. Each quotient has genus six by Riemann--Hurwitz. The classification of periodic surface maps by quotient genus and rotation data \cite{Nielsen} therefore identifies $Q$ and $R$, separately, with conjugates of $S^{-1}$ in $\Mod(F)$. Their relative conjugators depend on transport between the invariant fibers and do not determine $A$. Thus the order-three equations in Theorem~\ref{thm:marked} are necessary but do not by themselves identify the Cartwright--Steger monodromy.

\section{Global transport}

Transport along $\gamma$ identifies the vanishing cycle on a nearby smooth fiber with $c\subset F$. Once $c$ is known on the marked fiber, Theorem~\ref{thm:crossing} computes $[c]\cdot S_*[c]$ from the crossings of $\pi(c)$ and their covering monodromy. Corollary~\ref{cor:real-translate} gives a second computation of its parity. The handle monodromy $A$ requires transport around a handle generator of $\pi_1(E)$.

The pencil of Proposition~\ref{prop:global-pencil} has degree $72$ on all but finitely many fibers; on $F$ it factors through the degree-three quotient and the degree-$24$ map. Computing $c$ requires continuation of its branch values and sheets along $\gamma$, with a change of projection or chart at exceptional Albanese values. The symmetry carries fibers along $\gamma$ and restricts to an automorphism only at invariant values.

At the real node, two branch values of the pencil coalesce. Their local behavior is as follows.

\begin{proposition}\label{prop:local-branch-pair}
Let $\theta=\sqrt[3]{1116}>0$, $q=6267859/1459^2$, and $\kappa=4q^2\theta/899$. In the affine chart $U_0=1$ at the real node, let
\[
u=U_8+\theta/2,\quad v=U_9-\theta,\qquad
z_1=2u+v,\quad z_2=41u+36v.
\]
Regarding $y$ as the Weierstrass coordinate on $E$, the Albanese coordinate $y\circ\alpha$ and the pencil function $h=L_0/L_1$, both restricted to $X$, have the following expansions in these local coordinates:
\[
y-y_*=-\kappa z_1z_2+O\bigl(\|(z_1,z_2)\|^3\bigr),\qquad
h=\frac{2}{3\theta}\left(z_1-\frac{z_2}{29}\right)
  +O\bigl(\|(z_1,z_2)\|^2\bigr).
\]
Consequently, for small $\delta=y-y_*>0$, the restriction of $h$ to the smooth fiber has two simple local branch values
\[
h_\pm=\pm\frac{4}{3\theta}\sqrt{\frac{\delta}{29\kappa}}+O(\delta).
\]
\end{proposition}
\begin{proof}
The intrinsic Hessian of $y$ at the node computed in \cite[Section~5]{Szabo} gives the first displayed expansion. Differentiating the Borisov--Yeung equations at the node gives
\[
dU_7=-\frac{12}{29}\,du-\frac{36}{29}\,dv.
\]
The calculation is also checked by the exact script \texttt{check\_node\_tangent.py} in the ancillary file named in Section~7. Use the ordered list of $85$ bicanonical equations in the Borisov--Yeung ancillary file \texttt{MagmaCSMarch2018Main} (arXiv:1804.00737v6), numbering the equations from zero as in that list. In the chart $U_0=1$, substitute the exact node coordinates of \cite[Section~5]{Szabo} and work over $\Q[\theta]/(\theta^3-1116)$. Differentiate equations $13,74,45,57,2,80,56$ with respect to $U_1,\ldots,U_7$, keeping $U_8,U_9$ as local parameters. The determinant of this $7\times7$ coefficient matrix is
\[
-26504212164176875073986447937805414466800000\theta.
\]
It is nonzero, and exact solution of the seven linear tangent equations yields $dU_7=-(12\,du+36\,dv)/29$ as above. This gives $dL_0=(17\,du-7\,dv)/29=dz_1-dz_2/29$. Since $L_0$ vanishes and $L_1=3\theta/2$ at the node, this gives the second expansion. On $y=y_*+\delta$, the first equation has $z_1z_2=-\delta/\kappa+O(\delta^{3/2})$. Substituting $z_2=-\delta/(\kappa z_1)+O(\delta)$ into $h$ gives, to leading order, $2(z_1+\delta/(29\kappa z_1))/(3\theta)$. Its derivative vanishes at the two values $z_1=\pm\sqrt{\delta/(29\kappa)}+O(\delta)$, with nonzero second derivative. Substitution gives $h_\pm$ as stated.
\end{proof}

In the local two-sheeted component of the cover, the vanishing cycle is the inverse image of a short arc joining these branch values. To leading order, it has the parametrization
\[
z_1=\sqrt{\frac{\delta}{29\kappa}}e^{i\phi},\qquad
z_2=-\sqrt{\frac{29\delta}{\kappa}}e^{-i\phi},\qquad 0\leq\phi\leq2\pi;
\]
its image under $h$ traverses the interval between $h_-$ and $h_+$ twice. Identifying this cycle on $F$ requires transporting the two branch values and their associated sheets along $\gamma$.

In the Weierstrass coordinate of Section~3, the quotient of $E$ by its order-three action is the $y$-line. The orbifold points are $y=\infty$ and $y=\pm36i\sqrt3$, while the three nodal values project to the single real value $y_*$. Since $\ell\sigma=\ell$, the pair $(y,\ell)$ descends to a rational map
\[
X/\langle\sigma\rangle\dashrightarrow\Pone_y\times\Pone_{\ell}
\]
of generic degree $72$. On a smooth fiber where the pencil is basepoint-free, Riemann--Hurwitz gives total ramification degree $2(19)-2+2(72)=180$. On $F$ the degree-$72$ map factors as $F\xrightarrow{\pi}G\to\Pone$, and several ramification points can share a branch value. To compute the braid of branch values from $(y,\ell)$, one must track these coincident branch values near $O$, exclude contributions from the finite base locus of $\ell$, and account for the cyclic quotient singularities over the orbifold points. The resulting based homology matrix $T$ would give $m=1+\operatorname{tr}(S_*T)$ by Proposition~\ref{prop:single-trace}.

The equations in \cite{BY,Rito,AYreal} and the Albanese formula in \cite{Szabo} provide data for this calculation. The path $\gamma$ already defines $c\subset F$ by parallel transport, and Proposition~\ref{prop:local-branch-pair} determines its local branch pair. What remains is to continue that pair and its sheets along $\gamma$ to determine the isotopy class of $c$ on the marked fiber $F$. Consequently the intersection $[c]\cdot S_*[c]$ and the return maps along the handle loops, including the pushed-off real circle, remain uncomputed.

\section{Bicanonical curves and a product cover}

There is an invariant pencil that retains the three nodes while making the intersection with the Albanese fibers accessible. Put
\[
L_2=U_7-U_8,\qquad D_t=X\cap\{L_0+tL_2=0\},\quad t\in\mathbb P^1.
\]
Here the member at infinity is defined by $L_2=0$. Both sections have the same $\sigma$-character, so every member is $\sigma$-invariant. On the nodal line of \cite{Szabo} one has $(U_7,U_8,U_9)=(-s,-s,2s)$, so $L_0=L_2=0$. Hence every $D_t$ passes through the three critical points of $\alpha$.

\begin{proposition}\label{prop:symmetric-pencil}
The base scheme of this pencil has length $36$. It consists of $24$ reduced points and six points of length two, the latter being the fixed points of $\sigma$ outside $F$. A general member $D_t$ is smooth. It has genus $28$, and $\alpha|_{D_t}:D_t\to E$ has degree $72$ and total ramification degree $54$. For a general parameter, the six fixed points outside $F$ and the three nodes contribute nine simple ramification points. The remaining ramification has degree $45$ and is arranged in $15$ free $\sigma$-orbits, counted with multiplicity. Moreover, $D_t/\langle\sigma\rangle$ has genus $8$.
\end{proposition}
\begin{proof}
The intersection number of two bicanonical divisors is $(2K_X)^2=36$. The exact verification scripts and their input equations are in the ancillary archive \href{https://github.com/akhmedov-coder/C-S-Fibration-Monodromy-verification.zip/blob/9d9b388bd2cdd56a618da3eb4081414b02e679e1/C-S-Fibration-Monodromy-verification.zip}{\texttt{C-S-Fibration-Monodromy-verification.zip}}. They work over $\mathbb Q$ with the $85$ bicanonical equations of \cite{BY} in the coordinate order used by \cite{Szabo}. In the chart $U_7=1$, the two pencil equations give $U_8=1$ and $U_9=-2$. The script \texttt{gb\_weight2\_Q.py} computes that the quotient algebra has dimension $24$ and that the characteristic polynomial of multiplication by $U_0$ is squarefree of degree $24$, proving that these are reduced points. The remaining base scheme has $U_7=U_8=U_9=0$. A homogeneous Gr\"obner calculation shows that its intersection with $U_3=0$ is only the affine cone vertex; it therefore has no projective points there. In the chart $U_3=1$ its length is $12$. The six fixed points of $\sigma$ outside $F$ belong to this scheme. The tangent calculation in the ancillary verifier shows that the differentials of $L_0$ and $L_2$ are proportional there, so each local intersection has length at least two; it also checks $dL_0\ne0$ at all six points. They exhaust the length-$12$ scheme and a general pencil member is smooth there. It is smooth at the reduced base points as well, and Bertini's theorem applies elsewhere.

Adjunction and $K_X^2=9$ give $2g(D_t)-2=(K_X+2K_X)\cdot2K_X=54$, so $g(D_t)=28$. Also $D_t\cdot F=2K_X\cdot F=72$, and Riemann--Hurwitz for a map to $E$ gives ramification degree $54$. At a fixed point outside $F$, the nonzero differential of $L_0$ is in the weight-two direction, so the tangent of $D_0$ has weight one. Since $\sigma_E$ has weight two, equivariance forces the local degree of $\alpha|_{D_t}$ to be congruent to two modulo three. In the local coordinate $x=P_9/Q_9$ of \cite{Szabo}, $Q_9$ does not vanish at any of these six points. For $t=0$ the verifier computes the quadratic coefficient along the tangent to $D_0$ modulo the ideal of the six points and checks that it is a unit. Thus each point contributes simple ramification for general $t$. At a node, Section~6 gives $y-y_*=-\kappa z_1z_2+O(3)$ and $d(L_0+tL_2)=dz_1-(1+t)dz_2/29$. For general $t\ne-1$ the tangent is nonaxial, so the restriction of $\alpha$ again has local degree two. These are nine simple ramification points. The remaining ramification points occur in free $\sigma$-orbits of length three, giving the stated residual degree, with multiplicity. Finally $\sigma|_{D_t}$ has precisely the six fixed points outside $F$ for general $t$, and Riemann--Hurwitz for $D_t\to D_t/\langle\sigma\rangle$ gives $54=3(2g-2)+12$, whence $g=8$.
\end{proof}

The pencil in Proposition~\ref{prop:symmetric-pencil} has nonreduced base points. For a finite map to a product we use a general bicanonical pencil instead.

\begin{theorem}\label{thm:product-cover}
Let $H_0,H_1\in|2K_X|$ be general divisors meeting transversely at $36$ points, and let $\pi:Y\to X$ be the blowup of these points. The resolved pencil $h:Y\to\mathbb P^1$ and the Albanese map give a finite map
\[
f=(\alpha\pi,h):Y\longrightarrow E\times\mathbb P^1
\]
of degree $72$. If $A=E\times\{\mathrm{pt}\}$ and $B=\{\mathrm{pt}\}\times\mathbb P^1$, the ramification divisor $R$ and its pushforward branch cycle satisfy
\[
R\equiv5\pi^*K_X-\sum_{i=1}^{36}E_i,
\qquad f_*R\equiv180A+54B.
\]
Here $e(Y)=39$ and $\operatorname{sign}(Y)=-35$.
\end{theorem}
\begin{proof}
The bicanonical system is very ample \cite{BY}, so a general pair has $(2K_X)^2=36$ distinct transverse base points, all away from the three singular fibers. Put $D=\pi^*(2K_X)-\sum E_i$, the class of a fiber of $h$, and write $F$ also for the pullback of a smooth Albanese fiber. Then $f^*A=D$, $f^*B=F$, and $D\cdot F=72$. Every Albanese fiber of $X$ is irreducible \cite{CKY}. The strict transform of a smooth fiber containing $k$ base points has degree $72-k\geq36$ under $h$, since these base points are transverse and have multiplicity one. A singular fiber contains no base points and has degree $72$. Each exceptional curve maps with degree one under $h$. These are all curves contracted by $\alpha\pi$, and none is contracted by $h$. The proper map $f$ is therefore finite, with degree $D\cdot F=72$.

Since $K_{E\times\mathbb P^1}=-2A$ and $K_Y=\pi^*K_X+\sum E_i$, the ramification formula gives $R=K_Y-f^*K_{E\times\mathbb P^1}=5\pi^*K_X-\sum E_i$. By the projection formula,
\[
(f_*R)\cdot B=R\cdot F=180,\qquad
(f_*R)\cdot A=R\cdot D=5(2K_X^2)-36=54.
\]
These intersections give the asserted numerical class of the branch cycle. Finally blowing up $36$ points increases the Euler characteristic by $36$ and decreases the signature by $36$.
\end{proof}

The branch cycle in Theorem~\ref{thm:product-cover} is taken with its ramification multiplicities; its numerical class makes no assertion that its support is smooth or that all branching is simple. There is also a homological restriction on any global monodromy obtained by continuation.

\begin{theorem}\label{thm:vanishing-span}
The homology classes of all monodromy translates of the three nodal vanishing cycles span $H_1(F;\Q)$.
\end{theorem}
\begin{proof}
Let $E^\circ$ be the complement of the three critical values, and let $W\subset H_1(F;\Q)$ be the invariant span of the vanishing classes under the monodromy of $E^\circ$. Put $K=\ker(\pi_1(X)\to\pi_1(E))$. Van Kampen for the nodal fibration identifies
\[
H_1(K;\Q)\cong H_1(F;\Q)/W.
\]
This is finite dimensional, and the conjugation action of $\pi_1(E)\cong\Z^2$ on it is well defined.

We use the presentation of $\pi_1(X)$ computed by Cartwright, Koziarz, and Yeung \cite{CKY}. In their generators $a_1,a_2,a_3$, the Albanese character sends them to $(1,3)$, $(-2,1)$, and $(-1,-1)$, respectively. Let $\mathcal J(s,t)$ be the Fox matrix of their thirteen relators, with these images substituted. Direct calculation gives
\[
\operatorname{rank}\mathcal J(s,t)=
\begin{cases}
1,&(s,t)=(1,1),\\
2,&(s,t)\in(\mathbb C^*)^2\setminus\{(1,1)\}.
\end{cases}
\]
Here is a short exact check of the second assertion. Number the relators as in the presentation file accompanying \cite{CKY}. Take the minors in columns $1,2$ and row pairs $(1,2),(1,3),(2,3),(1,4)$. After dividing by $st-1$ and clearing Laurent monomials, their polynomial factors have Gr\"obner basis $(s^4,t)$ in $\Q[s,t]$. Thus the rank is two when $st\ne1$. On $st=1$, the numerators of the minors in columns $1,3$ and row pairs $(1,2),(1,3)$, together with the minor in columns $2,3$ and rows $(1,2)$, have greatest common divisor $s-1$. The rank at $(1,1)$ follows by direct substitution. A script performing these calculations is included in the ancillary verification file.

For a nontrivial character $\chi$ of $\Z^2$, the degree-one boundary map of the presentation complex has rank one, so the Fox rank gives $H_1(X;\mathbb C_\chi)=0$. The homology five-term sequence for $1\to K\to\pi_1(X)\to\Z^2\to1$, with these twisted coefficients, then shows that the $\chi$-coinvariants of $H_1(K;\mathbb C)$ vanish. For the trivial character, the same sequence, together with the surjection $H_2(X;\Q)\to H_2(\pi_1(X);\Q)$, gives
\[
H_2(X;\Q)\longrightarrow H_2(E;\Q)\longrightarrow
H_1(K;\Q)_{\Z^2}\longrightarrow H_1(X;\Q)\longrightarrow H_1(E;\Q).
\]
The first arrow is onto: a bicanonical curve maps to $E$ with positive degree. The last arrow is an isomorphism. Hence the ordinary coinvariants vanish as well. If $H_1(K;\mathbb C)$ were nonzero, its finite-dimensional commuting $\Z^2$-action would have a one-dimensional character quotient. This contradicts the vanishing just proved for every character. Therefore $H_1(K;\Q)=0$, as required.
\end{proof}

\vspace{-0.5\baselineskip}
\begin{proposition}\label{prop:albanese-kernel-normal-generation}
Let $F$ be a smooth fiber of $\alpha:X\to E$, and let
$K=\ker(\alpha_*:\pi_1(X)\to\pi_1(E))$. The inclusion of $F$ induces
a surjection $\pi_1(F)\to K$. In terms of the generators $a_1,a_2,a_3$
in the presentation of Cartwright, Koziarz, and Yeung \cite{CKY}, $K$ is the
normal closure in $\pi_1(X)$ of
\[
[a_1,a_2],\qquad [a_1,a_3],\qquad [a_2,a_3],
\qquad a_1^{-3}a_2^{\,2}a_3^{-7}.
\]
\end{proposition}
\begin{proof}
The presentation of Cartwright, Koziarz, and Yeung gives generators
$a_1,a_2,a_3$ of $\pi_1(X)$ with images in
$\pi_1(E)\cong\mathbb Z^2$ given by
\[
v_1=(1,3),\qquad v_2=(-2,1),\qquad v_3=(-1,-1).
\]
Their pairwise determinants are $7,2,3$. Hence they generate
$\mathbb Z^2$, and $\alpha_*$ is onto. The kernel of the map
$\mathbb Z^3\to\mathbb Z^2$ with columns $v_1,v_2,v_3$ is generated
by $(-3,2,-7)$: this vector is primitive, and
$-3v_1+2v_2-7v_3=0$.

Let $\widetilde E\to E$ be the universal cover and
$\widetilde X=X\times_E\widetilde E$. Then
$\pi_1(\widetilde X)=K$. Over the complement of the critical values
the map $\widetilde X\to\widetilde E$ is a surface bundle. Around each
node, filling the puncture kills a peripheral lift, and the Lefschetz
thimble kills the transported vanishing cycle. Since
$\widetilde E$ is contractible, van Kampen applied over an exhaustion
by disks gives
\[
K\cong\pi_1(F)\big/
\big\langle\!\big\langle
\text{all transported nodal vanishing cycles}
\big\rangle\!\big\rangle.
\]
Thus $\pi_1(F)\to K$ is onto. Paths to the lifts of the critical values
include paths winding around the handle loops of $E$.

Let $N$ be the normal closure of the four elements in the statement.
They lie in $K$, so $N\subseteq K$. The quotient $\pi_1(X)/N$ is
abelian, hence a quotient of
$\mathbb Z^3/\langle(-3,2,-7)\rangle$. The map from this last group
to $\mathbb Z^2$ defined by $v_1,v_2,v_3$ is an isomorphism.
Therefore the map $\pi_1(X)/N\to\mathbb Z^2$ induced by $\alpha_*$
is an isomorphism. Its kernel is $K/N$, and $N=K$.
\end{proof}

The local pair in Proposition~\ref{prop:local-branch-pair} and the invariant curves of Proposition~\ref{prop:symmetric-pencil} offer two ways to approach the global braid calculation. Theorem~\ref{thm:vanishing-span} restricts its possible outcome, but does not identify the individual vanishing cycle or a handle return map. Determining the value of $m=[c]\cdot S_*[c]$ requires continuing the two sheets and their connecting vanishing arc to the marked fiber. We leave that global computation for further work.

\section*{Acknowledgments}
The authors used an LLM for grammatical editing and assistance with some computations.


\begin{thebibliography}{99}
\bibitem{AkhSmall} A. Akhmedov, \emph{Small exotic $4$-manifolds}, Algebr. Geom. Topol. \textbf{8} (2008), 1781--1794.
\bibitem{APSmall} A. Akhmedov and B.~D.~Park, \emph{Exotic smooth structures on small $4$-manifolds}, Invent. Math. \textbf{173} (2008), 209--223.
\bibitem{APOdd} A. Akhmedov and B.~D.~Park, \emph{Exotic smooth structures on small $4$-manifolds with odd signatures}, Invent. Math. \textbf{181} (2010), 577--603.
\bibitem{AYquotient} A. Akhmedov and S.-K. Yeung, \emph{Positive spheres and circle surgery in the real Cartwright--Steger quotient}, preprint, September 2026.
\bibitem{ASY} A. Akhmedov, S. Sakall\i{} and S.-K. Yeung, \emph{Complex ball quotients and new symplectic $4$-manifolds with nonnegative signatures}, Taiwanese J. Math. \textbf{28} (2024), 29--53.
\bibitem{AYreal} A. Akhmedov and S.-K. Yeung, \emph{The real Albanese map of the Cartwright--Steger surface}, preprint, September 2026.
\bibitem{BY} L. A. Borisov and S.-K. Yeung, \emph{Explicit equations of the Cartwright--Steger surface}, \`Epijournal de G\'eom\'etrie Alg\'ebrique \textbf{4} (2020), Article 10; ancillary file \url{https://arxiv.org/src/1804.00737v6/anc/MagmaCSMarch2018Main}.
\bibitem{CKY} D. I. Cartwright, V. Koziarz and S.-K. Yeung, \emph{On the Cartwright--Steger surface}, J. Algebraic Geom. \textbf{26} (2017), 655--689; long version, \url{https://www.math.u-bordeaux.fr/~vkoziarz/cs-surface_long.pdf}; presentation file, \url{https://www.maths.usyd.edu.au/u/donaldc/cs-surface/cs-surface-presentations_magma.txt}.
\bibitem{KY} V. Koziarz and S.-K. Yeung, \emph{Stability of the Albanese fibration on the Cartwright--Steger surface}, Taiwanese J. Math. \textbf{25} (2021), 251--256.
\bibitem{Nielsen} J. Nielsen, \emph{Die Struktur periodischer Transformationen von Fl\"achen}, Mat.-Fys. Medd. Danske Vid. Selsk. \textbf{15} (1937), no.~1, 1--77.
\bibitem{Rito} C. Rito, \emph{Surfaces with canonical map of maximum degree}, J. Algebraic Geom. \textbf{31} (2022), 127--135; ancillary computations, arXiv:1903.03017.
\bibitem{SS} A. I. Stipsicz and Z. Szab\'o, \emph{The real quotient of the Cartwright--Steger surface}, arXiv:2609.09207 (2026).
\bibitem{SY} A. I. Stipsicz and K.-H. Yun, \emph{On the minimal number of singular fibers in Lefschetz fibrations over the torus}, Proc. Amer. Math. Soc. \textbf{145} (2017), 3607--3616.
\bibitem{Szabo} Z. Szab\'o, \emph{A polynomial formula for the Albanese map of the Cartwright--Steger surface}, arXiv:2609.23124 (2026), with ancillary files.
\end{thebibliography}
\end{document}